\documentclass[12pt,reqno]{amsart}
\usepackage[a4paper]{geometry}
\usepackage{amssymb,amsmath,amsthm,enumerate,verbatim,bbm}
\usepackage[mathscr]{euscript}
\usepackage{xcolor}
\usepackage{mathrsfs}
\DeclareMathOperator{\Ran}{Ran}

\DeclareMathOperator{\Ker}{Ker}

\DeclareMathOperator{\diag}{diag}

\DeclareMathOperator{\Span}{span}
\DeclareMathOperator*{\Res}{Res}

\DeclareMathOperator{\BMO}{BMO}
\DeclareMathOperator{\BMOA}{BMOA}
\DeclareMathOperator{\VMOA}{VMOA}

\newcommand{\abs}[1]{{\lvert#1\rvert}}

\newcommand{\norm}[1]{\lVert#1\rVert}

\newcommand{\jap}[1]{\langle#1\rangle}

\newcommand{\bbT}{{\mathbb T}}
\newcommand{\bbR}{{\mathbb R}}
\newcommand{\bbC}{{\mathbb C}}
\newcommand{\bbN}{{\mathbb N}}
\newcommand{\bbZ}{{\mathbb Z}}
\newcommand{\bbD}{{\mathbb D}}

\newcommand{\calF}{{\mathcal F}}
\newcommand{\calH}{{\mathcal H}}

\newcommand{\calT}{\mathcal{T}}
\newcommand{\calC}{\mathcal{C}}

\newcommand{\scrC}{\mathscr{C}}

\newcommand{\bh}{\mathbf h}
\newcommand{\bone}{\mathbf 1}
\newcommand{\bu}{\mathbf u}
\newcommand{\bx}{\mathbf x}
\newcommand{\by}{\mathbf y}

\newcommand{\bpsi}{\pmb \psi}

\numberwithin{equation}{section}
\renewcommand{\[}{\begin{equation}}
\renewcommand{\]}{\end{equation}}

\theoremstyle{plain}
\newtheorem{theorem}{\bf Theorem}[section]
\newtheorem*{theorem*}{Theorem 1.1$'$}
\newtheorem{lemma}[theorem]{\bf Lemma}
\newtheorem{proposition}[theorem]{\bf Proposition}

\theoremstyle{definition}

\newtheorem*{definition*}{\bf Definition}

\theoremstyle{remark}
\newtheorem*{remark*}{\bf Remark}
\newtheorem{remark}[theorem]{\bf Remark}

\DeclareFontFamily{U}{mathx}{\hyphenchar\font45}
\DeclareFontShape{U}{mathx}{m}{n}{<5> <6> <7> <8> <9> <10>
<10.95> <12> <14.4> <17.28> <20.74> <24.88> mathx10}{}
\DeclareSymbolFont{mathx}{U}{mathx}{m}{n}
\DeclareFontSubstitution{U}{mathx}{m}{n}
\DeclareMathAccent{\wc}{0}{mathx}{"71}

\newcommand{\wt}{\widetilde}
\newcommand{\wh}{\widehat}
\newcommand{\eps}{\varepsilon}
\newcommand{\1}{\mathbbm{1}}
\newcommand{\proj}{P}
\newcommand{\sv}{{s}}
\newcommand{\svt}{{\wt s}}

\begin{document}

\title[Inverse theory for non-compact Hankel operators]{Inverse spectral theory for a class of non-compact Hankel operators}

\author{Patrick G\'erard}
\address{Universit\'e Paris-Sud XI, Laboratoire de Math\'ematiques d'Orsay, CNRS, UMR 8628, et Institut Universitaire de France}
\email{Patrick.Gerard@math.u-psud.fr}

\author{Alexander Pushnitski}
\address{Department of Mathematics, King's College London, Strand, London, WC2R~2LS, U.K.}
\email{alexander.pushnitski@kcl.ac.uk}



\begin{abstract}
We characterize all bounded Hankel operators $\Gamma $ such that $\Gamma^*\Gamma$ has finite spectrum. 
We identify spectral data corresponding to such operators and construct inverse spectral theory 
including the characterization of these spectral data.
\end{abstract}

\date{16 March 2018}

\maketitle

\section{Introduction and main results}\label{sec.a}

\subsection{Overview}

Let $\calH^2(\bbT)\subset L^2(\bbT)$ be the standard Hardy class, and 
let $\proj$ be the orthogonal projection onto $\calH^2(\bbT)$ in $L^2(\bbT)$ (the Szeg\H o projection). 
In \cite{GG00}, the \emph{cubic Szeg\H o equation} 
$$
i\frac{\partial u}{\partial t} =\proj(\abs{u}^2 u), \quad u=u(z;t), \quad z\in \bbT, \quad t\in \bbR,
$$
was introduced as a model for totally non-dispersive evolution equations. 
Here $u(\cdot,t)\in \calH^2(\bbT)$ for each $t$. 
It turned out \cite{GG00,GG0} that this equation is completely integrable and possesses a Lax pair
(in fact, two Lax pairs). 
The Lax pairs involve the Hankel operators $H_u$ and $H_{S^*u}$ 
corresponding to the symbols $u$ and $S^*u(z)=\frac{u(z)-u(0)}{z}$
($S^*$ is the standard backwards shift operator); 
precise definitions are given below. 
In particular, if the operators $H_u$ and $H_{S^*u}$ are compact, then the singular values of both 
operators are integrals of motion for the cubic Szeg\H{o} equation. 

In order to solve the Cauchy problem for the cubic Szeg\H{o} equation, 
one must therefore develop a certain version of direct and inverse spectral theory 
for the Hankel operators $H_u$ and $H_{S^*u}$.
The spectral data in this problem consists of the set of singular values of $H_u$, 
the set of certain inner functions, parameterising the Schmidt subspaces of $H_u$
(i.e. the eigenspaces of $\abs{H_u}$) and similar parameters for $H_{S^*u}$. 
This was achieved in \cite{GG1} for $u\in \VMOA(\bbT)$, 
which corresponds to \emph{compact} Hankel operators $H_u$ and $H_{S^*u}$. 

One of the ultimate aims in the study of the cubic Szeg\H{o} equation
is to understand the propagation of singularities. 
For example, one would like to understand the behaviour of the solution 
to the cubic Szeg\H{o} equation if the initial condition has a jump type 
singularity on the unit circle. 
As is well known, the operators $H_u$, $H_{S^*u}$  will be non-compact in this case. 
Thus, one faces the problem of extending the direct and inverse
spectral theory of \cite{GG1} to (bounded) non-compact $H_u$, $H_{S^*u}$. 
At the moment, this goal remains distant. 

In this paper, which has a methodological flavour, we make a first step in this direction. 
We consider bounded \emph{non-compact} Hankel operators $H_u$, $H_{S^*u}$
such that the spectra of $\abs{H_u}$, $\abs{H_{S^*u}}$ 
consist of finitely many eigenvalues (which are allowed to have infinite
multiplicities). 
This is certainly a very restrictive condition, which in particular excludes 
jump type singularities of $u$. 
However, this condition allows us to focus on the structure of the existing proofs and to significantly 
simplify and improve them in the following respects:
\begin{itemize} 
\item
Our proof of the surjectivity of the spectral map is achieved through a simple and direct algebraic calculation
(the proof of \cite{GG1} used a difficult indirect topological argument). 
\item
We provide a direct proof of the invertibility of a certain auxiliary operator 
(we call it a complex Cauchy matrix) which 
plays a central role in the construction. 
\item
In contrast to \cite{GG1}, our proof does not depend on compactness. 
\end{itemize}

We hope that the streamlined proofs presented in this paper will lead to further 
progress in this circle of problems.

\subsection{Hankel operators and Schmidt subspaces}
For general background on the spectral theory of Hankel operators, we refer to \cite{Nikolski,Peller}. 
For $u\in\BMOA(\bbT)=\BMO(\bbT)\cap\calH^2(\bbT)$, we consider the \emph{anti-linear} 
Hankel operator $H_u$ on the Hardy class $\calH^2(\bbT)$
defined by 
$$
H_u(f)=P(u\cdot\overline{f}), \quad f\in \calH^2(\bbT). 
$$
It is evident that the matrix of $H_u$ in the standard basis of $\calH^2(\bbT)$ is given by 
$$
(H_uz^n,z^m)=\wh u(n+m),
$$
where $u(z)=\sum_{n=0}^\infty \wh u(n)z^n$. Thus, $H_u$ is unitarily equivalent to the operator
$\Gamma\calC$ in $\ell^2(\bbZ_+)$, where $\Gamma$ is the Hankel matrix,
$\Gamma=\{\wh u(n+m)\}_{n,m=0}^\infty$, and $\calC$ is the operator of complex conjugation, 
$\calC\{x_n\}_{n=0}^\infty=\{\overline{x_n}\}_{n=0}^\infty$.
Considering anti-linear (rather than linear) Hankel operators is perhaps slightly non-standard, but
as we shall see, this approach has some advantages. 
In particular,  $H_u$ satisfies the symmetry relation 
\[
(H_uf,g)=(H_ug,f), \quad f,g\in \calH^2(\bbT), 
\label{n10}
\]
which implies that 
\[
(\Ran H_u)^\perp=\Ker H_u.
\label{n9}
\]
Further, it is easy to see that the square $H_u^2$ 
is a linear self-adjoint operator in $\calH^2(\bbT)$, 
which is unitarily equivalent to the operator $\Gamma\Gamma^*$ in $\ell^2$.

We also need Toeplitz operators on $\calH^2(\bbT)$; for a symbol $a\in L^\infty(\bbT)$, 
the Toeplitz operator $T_a$ on $\calH^2(\bbT)$ is defined by
$$
T_a(f)=P(a\cdot f). 
$$
In fact, we will use this definition for unbounded symbols $a$ as well, but the corresponding Toeplitz operators will turn out to be well defined
and bounded on suitable subspaces.

We recall (see e.g. \cite[Theorem 1.2.6]{Peller}) 
that for an inner function $\theta$, the operator $H_\theta^2$ is an orthogonal projection in $\calH^2(\bbT)$, and 
its range $\Ran H_\theta=\Ran H_\theta^2$ is the \emph{model space},
\[
\Ran H_\theta
=
\calH^2(\bbT)\cap(z\theta \calH^2(\bbT))^\perp
=
\{f\in \calH^2(\bbT): \theta\overline{f}\in \calH^2(\bbT)\}. 
\label{n11}
\]
The action of $H_\theta$ on $\Ran H_\theta$ is the simple involution 
$$
H_\theta f=\theta\overline{f}. 
$$

In \cite{A}, we have
described the eigenspaces of $H_u^2$ corresponding to non-zero eigenvalues.
Denoting 
$$
E_H(s)=\Ker(H_u^2-s^2I), \quad s>0,
$$
we have proved that for every bounded Hankel operator $H_u$, every non-zero subspace $E_H(s)$ can be described as
\[
E_H(s)=T_p\Ran H_\theta. 
\label{z0}
\]
Here $\theta$ is an inner function and  $p\in\calH^2(\bbT)$ is an \emph{isometric multiplier} on $\Ran H_\theta$. 
This means that $T_p$ acts isometrically on $\Ran H_\theta$; in other words,
$pf\in \calH^2(\bbT)$ for every $f\in \Ran H_\theta$ and 
$$
\norm{pf}=\norm{f}, \quad f\in \Ran H_\theta. 
$$
Furthermore, the action of $H_u$ on $\Ran H_\theta$ is given by 
\[
H_uT_p=sT_pH_\theta, \quad \text{ on $\Ran H_\theta$.}
\label{z00}
\]

\subsection{The operator $K_u$}
The shift operator $S$ on $\calH^2(\bbT)$ is defined as
$$
Sf(z)=zf(z),\quad  z\in \bbT\ .
$$
Next, for $u\in\BMOA(\bbT)$ the Hankel operator $H_u$ satisfies 
 \[
H_uS=S^*H_u=H_{S^*u}\ .
\label{a4a}
\]
We set $K_u:=H_{S^*u}$.  
We have a crucial identity
\[
K_u^2=H_uSS^*H_u=H_u^2-(\cdot,u)u\ ,
\label{a5}
\]
where $(\cdot, u)u$ denotes the rank one operator corresponding to $u$, considered as an element
of $\calH^2(\bbT)$. 
For $s>0$, similarly to $E_H(s)$, we denote
$$
E_K(s)=\Ker(K_u^2-s^2I). 
$$
We recall the basic statement which shows that (as a consequence of \eqref{a5})
the eigenspaces $E_H(s)$ and $E_K(s)$ differ by the one-dimensional subspace 
spanned by $u$. 
\begin{proposition}\label{lma.z1}\cite[Lemma 2.1]{A}
Let $s>0$ be a singular value of either $H_u$ or $K_u$, i.e. 
$$
E_H(s)+E_K(s)\not=\{0\}.
$$
Then one (and only one) of the following properties holds: 
\begin{enumerate}[\rm (1)]
\item
$u\not\perp E_H(s)$, and $E_K(s)=E_H(s)\cap u^\perp$;
\item
$u\not\perp E_K(s)$, and $E_H(s)=E_K(s)\cap u^\perp$.
\end{enumerate}
\end{proposition}
In case (1) above, we will say that 
the singular value $s$ is \emph{$H$-dominant};
in case (2) we will say that $s$ is 
\emph{$K$-dominant}.

\subsection{ The Schmidt subspaces of $H_u$ and $K_u$}\label{sec.a6}

The following theorem  describes the Schmidt subspaces of both $H_u$ and $K_u$. 
The two cases below correspond to the two cases in Proposition~\ref{lma.z1}. 
We denote by $\1$ the function on $\calH^2$ identically equal to $1$.

\begin{theorem}\label{thm.z1}\cite[Theorem~2.2]{A}
Let $u\in\BMOA(\bbT)$ and let the anti-linear Hankel operators $H_u$, $K_u$ be as defined above. 
Let $s>0$ be a singular value of either $H_u$ or $K_u$. 
\begin{enumerate}[\rm (1)]
\item
Let $s$ be $H$-dominant, and let $P_s$ be the orthogonal projection onto $E_H(s)$. 
Then \eqref{z0}, \eqref{z00} hold with some
inner function $\theta=\psi_s$ and with $p=P_s\1/\norm{P_s\1}$: 
\[
E_H(s)=p\Ran H_{\psi_s}, \quad H_uT_p=sT_pH_{\psi_s} \text{ on $\Ran H_{\psi_s}$,} \quad p=\frac{P_s\1}{\norm{P_s\1}}.
\label{a5a}
\]
Furthermore, $\psi_s$ is uniquely defined by these conditions. 
\item
Let $s$ be $K$-dominant, and let $\wt P_s$ be the orthogonal projection onto $E_K(s)$. 
Then for an inner function $\wt\psi_s$ we have
\[
E_K(s)=p\Ran H_{\wt\psi_s}, \quad K_uT_{p}=sT_{p}H_{\wt \psi_s} \text{ on $\Ran H_{\wt\psi_s}$,} \quad p=\frac{\wt P_s u}{\norm{\wt P_s u}}.
\label{a5b}
\]
Furthermore, $\wt\psi_s$ is uniquely defined by these conditions. 
\end{enumerate}
\end{theorem}

\subsection{Finite spectrum: direct spectral problem}\label{sec.a7}
In this paper we 
focus on the case when $H_u^2$ has finite spectrum. 
By the spectral theorem for self-adjoint operators, this spectrum is made of eigenvalues. 
Furthermore, by the identity  \eqref{a5}, 
the spectrum of $K_u^2$ is finite as well. 
In this subsection, we explain how to divide the corresponding singular values into two groups 
($H$-dominant and $K$-dominant) and state an important interlacing property for them.

We first note that
for any Hankel operator $H_u$, zero is in the spectrum
of $H_u^2$ and $\dim\Ker H_u$ is either zero or infinity; 
this is well known and is a consequence of the commutation relation 
\eqref{a4a} and of Beurling's theorem. 
Thus, in the finite spectrum case both $\Ker H_u$ and $\Ker K_u$ 
are infinite dimensional. 
Also, since $0$ is an isolated point in the spectrum of $H_u^2$, 
the range $\Ran H_u=\Ran H_u^2$ is closed.

Observe that by \eqref{a5}, we have $\Ker H_u\subset \Ker K_u$ and therefore (by \eqref{n9})
$\Ran K_u\subset \Ran H_u$. 
It follows that the orthogonal projections onto $\Ran K_u$ and $\Ran H_u$ commute;
in particular, $\Ran H_u$ is an invariant subspace
both for $H_u$ and $K_u$. 
It will be convenient to consider $H_u$ and $K_u$ restricted to this invariant subspace; 
this point of view eliminates the kernel of $H_u$, but
distinguishes between the cases $\Ker H_u=\Ker K_u$ and $\Ker H_u\subsetneq\Ker K_u$. 
In accordance with this, we 
extend the definition of $E_H(s)$ and $E_K(s)$ to 
$s=0$ by setting 
\[
E_H(0)=\{0\}, \quad E_K(0)=\Ker K_u\cap\Ran H_u.
\label{d0}
\]
Next, we define the sets of $H$-dominant and $K$-dominant singular values by 
\begin{align*}
\Sigma_H(u)&=\{s>0: u\not\perp E_H(s)\},
\\
\Sigma_K(u)&=\{s\geq0: u\not\perp E_K(s)\}.
\end{align*}
We note that $0\in\Sigma_K(u)$ if and only if $u\not\perp \Ker K_u$. 

By Proposition~\ref{lma.z1}, the sets $\Sigma_H(u)$ and $\Sigma_K(u)$ 
are disjoint and 
$$
\sigma(\abs{H_u})\cup \sigma(\abs{K_u})\subset
\Sigma_H(u)\cup\Sigma_K(u)\cup\{0\}. 
$$
The following set of properties was first observed in \cite{GG1} in the case of compact $H_u$, $K_u$. 
\begin{theorem}\label{lma.d1}
Assume that $u\in\BMOA(\bbT)$ is not identically zero and that the spectrum of $H_u^2$ is finite. Then:
\begin{enumerate}[\rm (i)]
\item
For some $N\in\bbN$, we have
$$
\Sigma_H(u)=\{\sv_1,\dots,\sv_N\}, \quad 
\Sigma_K(u)=\{\svt_1,\dots,\svt_N\}, 
$$
where the elements $\sv_j$, $\svt_k$ can be enumerated so that 
\[
\sv_1>\svt_1>\sv_2>\svt_2>\cdots>\sv_N>\svt_N\geq0.
\label{d1}
\]
\item
Denote by $u_j$ the orthogonal projection of $u$ onto $E_H(\sv_j)$ and by $\wt u_k$ the 
orthogonal projection of $u$ onto $E_K(\svt_k)$. 
Then for every $1\leq j,k\leq N$, 
\begin{align}
u_j&=\norm{u_j}^2\sum_{k=1}^N \frac{\wt u_k}{\sv_j^2-\svt_k^2},
\label{d2a}
\\
\wt u_k&=\norm{\wt u_k}^2\sum_{j=1}^N \frac{u_j}{\sv_j^2-\svt_k^2}.
\label{d2b}
\end{align}
\item
We have
\begin{align}
\Sigma_K(u)&=\{\svt\geq0,\  \svt\notin\Sigma_H(u): \sum_{j=1}^N \frac{\norm{u_j}^2}{\sv_j^2-\svt^2}=1\},
\label{d3}
\\
\Sigma_H(u)&=\{\sv>0, \  \sv\notin\Sigma_K(u): \sum_{k=1}^N \frac{\norm{\wt u_k}^2}{\sv^2-\svt_k^2}=1\}.
\label{d4}
\end{align}
\end{enumerate}
\end{theorem}
The proof is given in Section~\ref{sec.B}. 

\subsection{Main result: inverse spectral problem}\label{sec.a7a}

We introduce the spectral data
\[
\Lambda(u)=\bigl(\{\sv_j\}_{j=1}^N, \{\svt_k\}_{k=1}^N, \{\psi_j\}_{j=1}^N, \{\wt \psi_k\}_{k=1}^N\bigr),
\label{z2a}
\]
where $\psi_j=\psi_{\sv_j}$, $\wt \psi_k=\wt \psi_{\svt_k}$ are the inner 
functions from Theorem~\ref{thm.z1}. 
If $\wt s_N=0$, then
$\wt\psi_N$ does not enter spectral data (in all expressions below, 
it will appear in combinations $\wt s_N\wt\psi_N$).

Our main result below shows that 
the map 
$u\mapsto \Lambda (u)$
is a bijection with an explicit inverse. 
Below $\bbD$ is the unit disk in the complex plane.

\begin{theorem}\label{inverse}
\begin{enumerate}[\rm (i)]
\item
Let $u\in \BMOA(\bbT )$, $u\not\equiv0$, be such that the spectrum of $H_u^2$ is finite. 
Then $u$ is uniquely determined by the spectral data $\Lambda(u)$ 
according to the following formulas. 
Let $u_j$, $\wt u_k$ be as in Theorem~\ref{lma.d1} and let 
$h_j=\sv_j^{-1}H_u u_j$. 
For every $z\in \bbD$, consider the $N\times N$ matrix
\[
\scrC(z)=\left \{\frac{\sv_j -z\svt _k\psi_j(z)\wt \psi_k(z)}{\sv_j^2-\svt_k^2}\right \}_{1\leq j,k\leq N}\ .
\label{d7a}
\]
Then 
\begin{gather}
\bh(z)=(\scrC(z)^\top)^{-1}\bone, 
\quad \bh=(h_1, \dots, h_N)^\top, \quad \bone=(1,\dots, 1)^\top, 
\label{d7b}
\\
\wt \bu(z)= \scrC(z)^{-1}\bpsi(z), 
\quad \wt \bu=(\wt u_1, \dots, \wt u_N)^\top, \quad \bpsi=(\psi_1, \dots, \psi_N)^\top. 
\label{d7c}
\end{gather}
In particular, 
\[
u(z)= \jap{\scrC(z)^{-1}\bpsi(z), \bone}_{\bbC^N}. 
\label{d7}
\]
\item
Conversely, for $N\in \bbN$,  let $\{\sv_j\}_{j=1}^N$, $\{\svt_k\}_{k=1}^N$ be real numbers satisfying the interlacing
condition \eqref{d1}, 
and let $\{\psi_j\}_{j=1}^N$ and $\{\wt\psi_k\}_{k=1}^N$ be arbitrary inner functions.
Then there exists a unique symbol $u\in \BMOA(\bbT)$ with the spectral data 
$$
\Lambda(u)=\bigl(\{\sv_j\}_{j=1}^N, \{\svt_k\}_{k=1}^N, \{\psi_j\}_{j=1}^N, \{\wt \psi_k\}_{k=1}^N\bigr).
$$
In fact, the symbol $u(z)$ is a bounded rational function of $z\in\bbD$, $\psi_j(z)$, $\wt\psi_k(z)$, 
given by \eqref{d7}. 
\end{enumerate}
\end{theorem}

In Section~\ref{sec.a8} we will see that the matrix $\scrC(z)$ is invertible for all $\abs{z}\leq 1$.

\subsection{Related results and alternative methods}

In \cite{MPT}, a general inverse spectral problem for bounded self-adjoint Hankel matrices $\Gamma=\{\gamma_{j+k}\}_{j,k=0}^\infty$ 
was solved. 
The spectral data was taken to be the spectral multiplicity function of $\Gamma$.  
This spectral data is incomplete (it does not uniquely define $\Gamma$), and no explicit
formulas for a solution to this problem was discussed. 
The main task of \cite{MPT} was to determine the class of possible spectral multiplicity functions. 

In \cite{GG1}, an inverse problem for \emph{compact} Hankel operators with $H_u^2$, $K_u^2$ 
having simple spectrum (when restricted to $\overline{\Ran H_u}$)
was solved. The spectral data was similar to \eqref{z2a}, with $\psi_j$, $\wt\psi_k$ reducing to unimodular constants in this case. 
An explicit formula for $u$ in terms of the spectral data was given, but it had a more complicated form than \eqref{d7}. 

In \cite{Poc}, a similar question was considered for \emph{compact} Hankel operators acting in the Hardy 
space on the real line. 

The paper \cite{GP} was devoted to the case of bounded non-compact positive self-adjoint Hankel operators.
In this case, the inner functions $\psi_j$, $\wt\psi_k$ in the spectral data degenerate (they are equal to $1$) and the 
sequences of singular values $\{s_j\}$, $\{\wt s_k\}$ must be replaced by certain spectral measures.

In \cite{GGAst}, the general case of \emph{compact} Hankel operators $H_u$, $K_u$ was solved. 
The spectral data was of the same type as considered in our paper. 
The inverse formula \eqref{d7} appeared in \cite{GGAst} for the first time. 
The proof of the crucial statement  of the surjectivity of the spectral map $u\mapsto\Lambda(u)$   in \cite{GGAst} 
was very difficult. It was a combination of the following components:
(i) reduction to a finite rank case by the use of the AAK theorem; 
(ii) the image of $\Lambda$ is both open and closed, the latter being obtained by a compactness argument;
(iii) for a given set of multiplicities of the singular values, the relevant set of $u$'s is non-empty; this was
achieved through an induction on multiplicity argument.

In \cite{G}, a different approach to the proof of surjectivity was presented in the case
of \emph{finite rank} $H_u$, $K_u$ with all singular values being simple. 
It was based on an adaptation of the asymptotic completeness argument of \cite{MPT}, 
taking into account the finite dimensionality of the problem. 

A direct proof of Theorem~\ref{inverse} in the very special case $N=1$ was provided as
an Example in \cite[Section 6]{A}.

\subsection{Some ideas of the proof}\label{sec.a8}
The proofs of Theorem~\ref{lma.d1} and part (i) of Theorem~\ref{inverse} follow closely 
the ideas of \cite{GGAst}, and present no fundamental new challenges. 
On the contrary, the proof of part (ii) of Theorem~\ref{inverse} is the main new feature of this paper. 
In our setting, $H_u$ and $K_u$ need not be compact and the 
multiplicities of the singular values may be infinite. Because of this, none of the methods mentioned above work for the proof of 
surjectivity.  
We suggest a new approach, which is more direct and algebraic in nature. 
For $u$ given by \eqref{d7}, we check directly that the operators $H_u$ and $K_u$ have the required
spectral data. A crucial component of this is the invertibility of the matrix $\scrC(z)$ for $z$ in the closed unit disk $\overline{\bbD}$. 
This depends on the following elementary statement about complex Cauchy matrices. 
\begin{theorem}\label{thm.z4}
Let $\{\sv_j\}_{j=1}^N$, $\{\svt_k\}_{k=1}^N$ be real numbers
satisfying the interlacing condition \eqref{d1}. Let $z, \zeta_1,\dots,\zeta_N, \wt \zeta_1, \dots, \wt \zeta_N$ 
be complex numbers in the closed unit disk $\overline{\bbD}$, and let
$\scrC(z;\zeta;\wt \zeta)$ be the $N\times N$ complex Cauchy matrix
\[
\scrC(z;\zeta;\wt\zeta)=
\left \{\frac{\sv _j-z\svt _k\zeta _j\wt \zeta _k}{\sv_j^2-\svt _k^2}\right\}_{1\leq j,k\leq N}\ .
\label{A0}
\]
Then $\scrC(z;\zeta;\wt\zeta)$
is invertible for all values of $z$, $\zeta_j$, $\wt \zeta_k$ in $\overline{\bbD}$, and the 
norm of the inverse is uniformly bounded:
$$
\norm{\scrC(z;\zeta;\wt\zeta)^{-1}}\leq C, \quad 
z, \zeta_j, \wt \zeta_k\in\overline{\bbD},
$$
where $C$ depends only on $\{\sv_j\}$, $\{\svt_k\}$. 
\end{theorem}
Of course, the uniform bound on the inverse follows immediately by a compactness argument; 
the main point is the invertibility of $\scrC$. 

Observe that for $z=1$, $\zeta_j=1$, $\wt \zeta_k=1$, the matrix $\scrC$ reduces to the
classical Cauchy matrix
$$
\left\{\frac1{\sv_j+\svt_k}\right\}
$$
which is known to be invertible if all numbers $\{\sv_j\}$, $\{-\svt_k\}$ are distinct.

\subsection{The structure of the paper}
In Section~\ref{sec.A} we prove Theorem~\ref{thm.z4};  the proof is elementary in nature but not very short. 
In Section~\ref{sec.B}, following the argument of \cite{GGAst}, we prove Theorem~\ref{lma.d1} and Theorem~\ref{inverse}(i). 
In Section~\ref{sec.C} we prove Theorem~\ref{inverse}(ii); this is the main part of the paper.

\subsection{Acknowledgements}
Our research was supported by EPSRC grant ref. EP/N022408/1. 
We acknowledge the hospitality of the Mathematics Departments at King's College and at Universit\'e Paris-Sud XI. 
We are grateful to Nikolai Nikolski for useful discussions.

\section{Complex Cauchy matrices}\label{sec.A}

The aim of this section is to prove Theorem~\ref{thm.z4}. 
Much of the analysis of this section can be viewed as an extension of the classical fact of 
the invertibility of the Cauchy matrix; this will be explained below.

\subsection{Classical Cauchy matrix}
The classical Cauchy matrix is an $N\times N$ matrix of the form 
$$
\left\{\frac{1}{a_j-b_k}\right\}_{j,k=1}^N\ ,
$$
where $a_1,\dots,a_N,b_1,\dots,b_N$ are \emph{distinct} complex numbers. 
Under this assumption, the Cauchy matrix is invertible. We need to develop 
some further algebra related to  Cauchy matrices. 
We are interested in the case $a_j=\sv_j^2$, $b_k=\svt_k^2$, so we denote 
\[
\calT=\{\calT_{jk}\}_{j,k=1}^N, \quad \calT_{jk}=\frac{1}{\sv_j^2-\svt_k^2}. 
\label{A6}
\]
For a sequence $\alpha=\{\alpha_1,\dots,\alpha_N\}$, we denote by $D(\alpha)$ the diagonal $N\times N$ matrix
with elements $\{\alpha_1,\dots,\alpha_N\}$ on the diagonal. In particular, we will make use of the diagonal matrices
$D(\sv)$, $D(\svt)$.
 
The following lemma is well-known, but for completeness we give a proof. 

\begin{lemma}\label{lma.A1}
Assume the interlacing condition \eqref{d1}; then the Cauchy matrix \eqref{A6} is invertible, and 
its inverse is given by 
\begin{gather}
\calT^{-1}=D(\varkappa^2)\calT^\top D(\tau^2), \quad \text{ where }
\label{A7}
\\
\tau_j^2:=\frac{\prod_{k=1}^N (\sv_j^2-\svt_k^2)}{\prod_{i\not=j}(\sv_j^2-\sv_i^2)}>0,
\quad
\varkappa_k^2:=\frac{\prod_{j=1}^N (\sv_j^2-\svt_k^2)}{\prod_{\ell\not=k}(\svt_\ell^2-\svt_k^2)}>0.
\label{A8}
\end{gather}
\end{lemma}
\begin{proof}
1)
Denote 
$$
A(z)=\prod_{i=1}^N (z-\sv_i^2), 
\quad
B(z)=\prod_{\ell=1}^N (z-\svt_\ell^2).
$$
Observe that 
\begin{align*}
\Res_{z=\svt_k^2}\frac{A(z)}{B(z)}&=\frac{A(\svt_k^2)}{B'(\svt_k^2)}=-\varkappa_k^2,
\\
\Res_{z=\sv_j^2}\frac{B(z)}{A(z)}&=\frac{B(\sv_j^2)}{A'(\sv_j^2)}=\tau_j^2.
\end{align*}
Thus, for $z\in \bbC$ we have
\begin{align}
\frac{A(z)}{B(z)}&=1-\sum_{k=1}^N\frac{\varkappa_k^2}{z-\svt_k^2},
\label{A13}
\\
\frac{B(z)}{A(z)}&=1+\sum_{j=1}^N\frac{\tau_j^2}{z-\sv_j^2}.
\label{A14}
\end{align}
An elementary inspection of these rational functions on the real axis shows that 
the interlacing condition \eqref{a5a} implies the positivity of $\tau_j^2$, $\varkappa_k^2$. 

2)
Let us compute the inverse of $\calT$. 
Let $\calT \bx=\by$ for some $\by\in \bbC^N$:
\[
\sum_{k=1}^N \frac{x_k}{\sv_j^2-\svt_k^2}=y_j,
\quad j=1,\dots,N.
\label{A11}
\]
Consider the rational function 
$$
R(z)=\sum_{k=1}^N \frac{x_k}{z-\svt_k^2}=\frac{P(z)}{B(z)}, 
$$
where $P$ is a polynomial of degree $\leq N-1$, whose definition is clear
from this formula. 
Condition \eqref{A11} can be written as $R(\sv_j^2)=y_j$, or equivalently as
$$
P(\sv_j^2)=B(\sv_j^2)y_j,\quad j=1,\dots,N,
$$
since $B(\sv_j^2)\not=0$ for all $j$. This allows us to recover $P(z)$ 
from its values at $\sv_j^2$ 
through the
Lagrange interpolation formula:
$$
P(z)
=
\sum_{j=1}^N P(\sv_j^2)\frac{A(z)}{A'(\sv_j^2)(z-\sv_j^2)}
=
\sum_{j=1}^N y_j\frac{B(\sv_j^2)A(z)}{A'(\sv_j^2)(z-\sv_j^2)}
=
\sum_{j=1}^N y_j \tau_j^2 
\frac{A(z)}{(z-\sv_j^2)}.
$$
Now we can determine $x_k$ as the residue of $R(z)$ at $z=\svt_k^2$: 
$$
x_k
=
\Res_{z=\svt_k^2}\frac{P(z)}{B(z)}
=
\frac{P(\svt_k^2)}{B'(\svt_k^2)}
=
\sum_{j=1}^N y_j \tau_j^2 
\frac{A(\svt_k^2)}{B'(\svt_k^2)(\svt_k^2-\sv_j^2)}
=
\sum_{j=1}^N y_j  
\frac{\tau_j^2 \varkappa_k^2 }{\sv_j^2-\svt_k^2}.
$$
This shows that $\bx=D(\varkappa^2)\calT^\top D(\tau^2)\by$, 
which proves that $\calT$ is invertible and the inverse is
given by \eqref{A7}, \eqref{A8}.
\end{proof}

In what follows, we will assume that the interlacing condition \eqref{d1} holds, and so 
all parameters $\varkappa_k^2$ and $\tau_j^2$ are positive. We set 
$\varkappa_k=\sqrt{\varkappa_k^2}>0$ and $\tau_j=\sqrt{\tau_j^2}>0$ and make use
of the diagonal operators $D(\varkappa)$, $D(\tau)$. 

\begin{lemma}
Assume the interlacing condition \eqref{d1}. 
Then the operator
\[
V:=D(\tau)\calT D(\varkappa)
\label{A17}
\]
is unitary in $\bbC^N$. 
\end{lemma}
\begin{proof}
Multiplying formula \eqref{A7} for the inverse of $\calT$ 
by $D(\varkappa)^{-1}$ on the left and by $D(\tau)^{-1}$ on the right, we get
$$
D(\varkappa)^{-1}\calT^{-1} D(\tau)^{-1}
=
D(\varkappa)\calT^\top D(\tau),
$$
which can be rewritten as $V^{-1}=V^\top$. Since $V$ has real entries, we get 
$V^{-1}=V^*$, and so $V$ is unitary. 
\end{proof}

Below we denote by $\jap{\cdot,\cdot}$ the standard inner product in $\bbC^N$. 
\begin{lemma}
Assume the interlacing condition \eqref{d1}. 
We have the identities
\begin{gather}
D(\sv^2)\calT-\calT D(\svt^2)=\jap{\cdot,\bone}\bone,\quad \bone=(1,\dots,1)^\top,
\label{A18}
\\
V^*D(\sv^2)V=D(\svt^2)+\jap{\cdot,\bx}\bx,\quad \bx=D(\varkappa)\bone,
\label{A19}
\end{gather}
where $V$ is the unitary operator \eqref{A17}. 
\end{lemma}
\begin{proof}
Identity \eqref{A18} follows directly from the definition of $\calT$. 
Multiplying \eqref{A18} by $D(\tau)$ on the left and by $D(\varkappa)$ on the right
and using the definition of $V$, we get
$$
D(\sv^2)V=VD(\svt^2)+\jap{\cdot,D(\varkappa)\bone}D(\tau)\bone
=
V\bigl(D(\svt^2)+\jap{\cdot,D(\varkappa)\bone}V^*D(\tau)\bone\bigr).
$$
Thus, it remains to check that 
$$
V^*D(\tau)\bone=D(\varkappa)\bone.
$$
By the definition of $V$, the last identity can be equivalently rewritten as
\[
D(\varkappa)\calT^{\top} D(\tau^2)\bone=D(\varkappa)\bone.
\label{A20}
\]
In order to prove this, set $z=\svt_k^2$ in \eqref{A14}. 
The left side vanishes, and we obtain
$$
1=\sum_{j=1}^N \frac{\tau_j^2}{\sv_j^2-\svt_k^2}, 
\quad k=1,\dots,N,
$$
which can be equivalently rewritten as
\[
\bone=\calT^{\top}  D(\tau^2)\bone.
\label{A20a}
\]
This proves \eqref{A20}.
\end{proof}

\subsection{Invertibility of $\scrC(z;\zeta;\wt\zeta)$ for $\abs{z}<1$}

We come back to the matrix $\scrC$ defined in \eqref{A0}. 
Let $\zeta_j$, $\wt\zeta_j$, $j=1,\dots,N$ be complex numbers in the 
closed unit disk $\overline{\bbD}$. 
Observe that $\scrC$ can be written as 
\[
\scrC=D(\sv)\calT-zD(\zeta)\calT D(\svt)D(\wt\zeta).
\label{A22}
\]

\begin{lemma}\label{lma.A5}
Assume the interlacing condition \eqref{d1}. 
Then the matrix $\scrC(z;\zeta;\wt\zeta)$ is invertible for $\abs{z}<1$. 
\end{lemma}
\begin{proof}
1) 
First let us check the estimate
\[
\norm{D(\sv)^{-1}VD(\svt)}\leq1.
\label{A21}
\]
Indeed,  identity \eqref{A19} shows that 
$$
VD(\svt^2)V^*=D(\sv^2)-\jap{\cdot,V\bx}V\bx
\leq D(\sv^2). 
$$
This yields
$$
D(\sv)^{-1}VD(\svt^2)V^*D(\sv)^{-1}\leq I,
$$
whence \eqref{A21} follows. 

2)
By \eqref{A22}, we can write $\scrC$ as 
$$
\scrC=D(\sv)\calT(I-zM)
=
D(\sv)\calT D(\varkappa)\bigl(I-zD(\varkappa)^{-1}MD(\varkappa)\bigr)D(\varkappa)^{-1},
$$
where 
$$
M=\calT^{-1}D(\sv)^{-1}D(\zeta)\calT D(\svt)D(\wt \zeta).
$$
It suffices to check that 
\[
\norm{D(\varkappa)^{-1}MD(\varkappa)}\leq1.
\label{A23}
\]
By the formula \eqref{A7} for the inverse of $\calT$, we have
\begin{multline*}
D(\varkappa)^{-1}MD(\varkappa)
=
D(\varkappa)\calT^{\top}  D(\tau^2)D(\sv)^{-1}D(\zeta)\calT D(\svt)D(\wt \zeta)D(\varkappa)
\\
=
V^*D(\sv)^{-1}D(\zeta)VD(\svt)D(\wt \zeta)
=
V^*D(\zeta)\bigl(D(\sv)^{-1}VD(\svt)\bigl)D(\wt \zeta).
\end{multline*}
By \eqref{A21}, the product in brackets in the right side here is a contraction. 
Thus, all terms in the right side are contractions, and so \eqref{A23} holds true.
\end{proof}

\subsection{Invertibility of $\scrC(z;\zeta;\wt\zeta)$ for $z$, $\zeta$, $\wt\zeta$ in $\bbT$}
Here we consider the invertibility of the matrix $\scrC(z;\zeta;\wt\zeta)$ for $z$, $\zeta$, $\wt\zeta$ in $\bbT$. 
Clearly, $z$ can be absorbed into $\zeta_j$ by a change of parameter. 
Next, multiplying $\scrC$ on the left by the diagonal matrix $D(\overline{\zeta})$, we arrive at a matrix of the form
$$
\left\{\frac{\overline{\zeta_j}\sv_j-\wt\zeta_k\svt_k}{\sv_j^2-\svt_k^2}\right\}_{1\leq j,k\leq N}.
$$
Changing notation, we shall discuss the invertibility of the matrix
\[
\calF=\left \{\frac{a_j-b_k}{|a_j|^2-|b_k|^2}\right \}_{1\leq j,k\leq N}\ ,
\label{A24}
\]
where  $a_1,b_1,\dots ,a_N,b_N$ are $2N$ complex numbers. 
Notice that, if all $a$'s and $b$'s are real, then $\calF$ reduces to the Cauchy matrix
$$
\left\{\frac{1}{a_j+b_k}\right\}\ .
$$

\begin{lemma}\label{main}
Assume that $a_1,b_1,\dots ,a_N,b_N$ are $2N$ complex numbers such that the moduli 
$|a_1|,|b_1|,\dots ,|a_N|,|b_N|$ are distinct. Then the matrix $\calF$ is invertible.
\end{lemma}
\begin{remark}
The invertibility of $\calF$ has been proved in \cite{GGAst} 
as a byproduct of the solution of an inverse spectral theorem for finite rank Hankel operators. Here we provide a short direct proof of this fact. 
\end{remark}
\begin{proof}
Assume for some non-zero vector $\bx=(x_1,\dots,x_N)^\top$, 
we have
$$
\sum_{k=1}^N\frac{a_j-b_k}{\abs{a_j}^2-\abs{b_k}^2}x_k=0, \quad j=1,\dots,N.
$$
Let us write this condition as
\[
z\sum_{k=1}^N\frac{x_k}{\abs{z}^2-\abs{b_k}^2}
=
\sum_{k=1}^N\frac{b_k x_k}{\abs{z}^2-\abs{b_k}^2}, 
\quad z=a_1,\dots, a_N.
\label{A1}
\]
Let us denote
$$
\frac{P_1(r)}{Q(r)}=\sum_{k=1}^N\frac{x_k}{r-\abs{b_k}^2}, 
\quad
\frac{P_2(r)}{Q(r)}=\sum_{k=1}^N\frac{b_k x_k}{r-\abs{b_k}^2},
\quad
Q(r)=\prod_{k=1}^N (r-\abs{b_k}^2),
$$
where $P_1$, $P_2$, are polynomials of degree $\leq N-1$ whose explicit form is evident from this definition.
Since $Q(\abs{z}^2)$ does not vanish for $z=a_j$ for any $j$, 
condition \eqref{A1} yields
\[
zP_1(\abs{z}^2)=P_2(\abs{z}^2), \quad z=a_1,\dots, a_N.
\label{A2}
\]
We claim that \eqref{A2} is also satisfied for $z=b_1,\dots,b_N$. 
Indeed, observe that, by the explicit form of  $P_1$ and $P_2$, we have
$$
P_1(\abs{b_j}^2)=x_j\prod_{k\not=j}(\abs{b_j}^2-\abs{b_k}^2), 
\quad
P_2(\abs{b_j}^2)=b_j x_j\prod_{k\not=j}(\abs{b_j}^2-\abs{b_k}^2), 
$$
and so we get \eqref{A2} at $z=b_j$.
 
Taking the square of the modulus on both sides of \eqref{A2}, we obtain
$$
r\abs{P_1(r)}^2=\abs{P_2(r)}^2, \quad r=\abs{z}^2.
$$
This is a polynomial equation in $r$ of degree $\leq 2N-1$, 
which is satisfied at $2N$ points $r=\abs{a_1}^2,\dots,\abs{a_N}^2, \abs{b_1}^2,\dots,\abs{b_N}^2$. 
Thus, it is satisfied for all $r\in\bbR$. 

Inspecting the coefficient in front of the highest degree of $r$ on the right and on the left of the 
last polynomial equation, 
we see that the left side has an odd degree in $r$ (or zero), whereas the right side has an even degree. 
We conclude that both sides must be identically zero. This implies $\bx=0$.
\end{proof}

\subsection{Proof of Theorem~\ref{thm.z4}}
Denote for brevity
$$
D(z;\zeta;\wt \zeta)
:= 
\det \scrC(z;\zeta;\wt \zeta)\ .
$$
The invertibility of $\scrC(z;\zeta;\wt \zeta)$ for $\abs{z}=\abs{\zeta_j}=\abs{\wt\zeta_k}=1$ follows from Lemma~\ref{main}, since the matrix 
$\calF=D(\overline{\zeta})\scrC$
is of the form \eqref{A24} with the choices 
$$
a_j=\overline{\zeta_j}\sv_j, \quad b_k=z\wt\zeta_k\svt_k.
$$
Thus, a compactness argument yields
$$
\sup_{\abs{z}=1,\abs{\zeta_j}=1,\abs{\wt\zeta_k}=1}1/\abs{D(z;\zeta;\wt \zeta)}=:D_0<\infty.
$$
Next, for every $\abs{z}<1$, from Lemma~\ref{lma.A5} we know that 
$D(z;\zeta ;\tilde \zeta )\ne 0$ for all $\zeta_j ,\wt\zeta_k\in\overline{\bbD}$.
Therefore, by the maximum modulus principle applied consecutively in each of the variables $\zeta_j,\wt\zeta_k$, we have
$$
\sup_{\abs{\zeta_j}\leq 1,\abs{\wt\zeta_k}\leq1}1/\abs{D(z;\zeta;\wt \zeta)}
=
\sup_{\abs{\zeta_j}=1,\abs{\wt\zeta_k}=1}1/\abs{D(z;\zeta;\wt \zeta)},\qquad \abs{z}<1,
$$
and so
\[
\sup_{\abs{z}<1}\sup_{\abs{\zeta_j}\leq 1,\abs{\wt\zeta_k}\leq1}1/\abs{D(z;\zeta;\wt \zeta)}
=
\sup_{\abs{\zeta_j}=1,\abs{\wt\zeta_k}=1}\sup_{\abs{z}<1}1/\abs{D(z;\zeta;\wt \zeta)}.
\label{n12}
\]
Further, for fixed $(\zeta,\wt\zeta)\in \bbT^{2N}$, 
applying the maximum modulus principle in the variable $z$, we obtain
$$
\sup_{\abs{z}<1}1/\abs{D(z;\zeta;\wt \zeta)}
=
\sup_{\abs{z}=1}1/\abs{D(z;\zeta;\wt \zeta)}\leq D_0. 
$$
Combining this with \eqref{n12}, we finally obtain 
$$
\sup_{\abs{z}\leq1}\sup_{\abs{\zeta_j}\leq 1,\abs{\wt\zeta_k}\leq1}1/\abs{D(z;\zeta;\wt \zeta)}\leq D_0<\infty.
$$
Thus, $\scrC(z;\zeta;\wt \zeta)$ is invertible for 
every $z,\zeta_j,\wt \zeta_k\in \overline{\bbD}$.
The estimate for the norm of the inverse follows immediately by a
compactness argument.

\section{Direct problem and formulas for the inverse spectral map}\label{sec.B}
In this section, we follow closely the argument of \cite{GGAst} and prove Theorems~\ref{lma.d1} and \ref{inverse}(i).

\subsection{Proof of Theorem~\ref{lma.d1}}

1) First observe that $\Sigma_H(u)$ cannot be empty, 
otherwise we get $u\perp \Ran H_u$ which is impossible since $u=H_u\1$. 
Denote the elements of $\Sigma_H(u)$ by $\{\sv_j\}_{j=1}^N$ and those of $\Sigma_K(u)$ by 
$\{\svt_k\}_{k=1}^{\wt N}$; initially we don't know that $N=\wt N$, but we will conclude this at the end of the proof. 
The spectral decompositions for the self-adjoint operators $H_u^2$ and $K_u^2$ and our definition \eqref{d0} yield
\[
\Ran H_u
=
\bigoplus_{s\in\sigma(\abs{H_u})}E_H(s)
=
\bigoplus_{s\in\sigma(\abs{K_u})}E_K(s).
\label{d5}
\]
Writing $u\in \Ran H_u$ according to these decompositions yields
$$
u=\sum_{j=1}^N u_j=\sum_{k=1}^{\wt N}\wt u_k.
$$

2)
Fix $1\leq k\leq\wt N$ and  and write $\wt u_k$ according to the first decomposition in \eqref{d5}.
Observe that $\wt u_k\perp E_H(\svt_\ell)$ for all $\ell$. 
Indeed:
\begin{itemize}
\item
if $\ell=k$, then $\wt u_k\perp E_H(\svt_k)$
by Proposition~\ref{lma.z1} (or by definition $E_H(0)=\{0\}$ if $\svt_k=0$). 
\item
if $\ell\not=k$, then 
$E_H(\svt_\ell)\subset E_K(\svt_\ell)$ and $\wt u_k\perp E_K(\svt_\ell)$.
\end{itemize}
Thus, we have a decomposition (recalling Proposition~\ref{lma.z1})
$$
\wt u_k=\sum_{j=1}^N(c_j u_j+f_j), \quad c_j\in \bbC, \quad f_j\in E_K(s_j). 
$$
But since $E_K(\wt s_k)\perp E_K(s_j)$, we conclude that $f_j=0$ for all $j$. 
It follows that we have the decomposition
\[
\wt u_k=\sum_{j=1}^N c_j u_j, \quad c_j=\frac{(\wt u_k,u_j)}{\norm{u_j}^2}. 
\label{n8}
\]
We have, using \eqref{a5}, 
$$
0=K_u^2 \wt u_k-\svt_k^2 \wt u_k=H_u^2 \wt u_k-\svt_k^2 \wt u_k-\norm{\wt u_k}^2u.
$$
Taking inner product with $u_j$, we obtain
$$
(\sv_j^2-\svt_k^2)(\wt u_k,u_j)=\norm{\wt u_k}^2\norm{u_j}^2.
$$
Substituting this into \eqref{n8} yields \eqref{d2b}. Formula \eqref{d2a} (with $\wt N$ instead of $N$ at this 
stage of the proof) is obtained in a similar way by expanding $u_j$ with respect to the second decomposition in \eqref{d5}. 

3) 
Denote by $\Sigma_K'(u)$ the right side of \eqref{d3} and by $\Sigma_H'(u)$ the right side of \eqref{d4}
(with $\wt N$ instead of $N$ at this stage of the proof). 
Take the inner product with $u$ on both sides of both equalities \eqref{d2a} and \eqref{d2b}; this gives the inclusions
$\Sigma_K(u)\subset \Sigma_K'(u)$ and 
$\Sigma_H(u)\subset \Sigma_H'(u)$. 

4) 
Let us check that $\Sigma_K'(u)\subset \Sigma_K(u)$. 
Take $\svt\in \Sigma_K'(u)$ and
\[
h=\sum_{j=1}^N \frac{u_j}{\sv_j^2-\svt^2};
\label{e2a}
\]
then, clearly, $h\not=0$ and 
$$
H_u^2h-\svt^2 h=\sum_{j=1}^N u_j=u.
$$
Taking the inner product of \eqref{e2a} with $u$, we obtain 
$$
(h,u)=\sum_{j=1}^N \frac{\norm{u_j}^2}{\sv_j^2-\svt^2}=1,
$$
and so, using \eqref{a5}, 
$$
K_u^2h-\svt^2 h=H_u^2 h-\svt^2 h-(h,u)u=u-u=0.
$$
It follows that $\svt^2\in\sigma(K_u^2)$ and $h\in E_K(\svt)$. 

If $\svt>0$, then recalling that $u\not\perp h$, we obtain $\svt\in\Sigma_K(u)$. 
If $\svt=0$, then, observing that $h\in \Ran H_u$, we have $h\in E_K(0)$ and $u\not\perp h$, and so again 
 $0\in \Sigma_K(u)$.

In a similar way, for $\sv\in\Sigma_H'(u)$ we check that 
$$
H_u^2h=\sv^2 h, \quad\text{ where }\quad
h=\sum_{k=1}^{\wt N} \frac{\wt u_k}{\sv^2-\svt_k^2}. 
$$
This proves that $\Sigma_H'(u)\subset \Sigma_H(u)$. 

5) 
Let $P_{s_j}$ be the orthogonal projection onto $E_H(s_j)$, and let $\1_{s_j}=P_{s_j}\1$. 
Since $H_u$ commutes with $P_{s_j}$, we have 
\[
u_j=P_{s_j}H_u\1=H_u \1_{s_j},
\label{n7}
\] 
and therefore 
$$
\norm{u_j}^2=\norm{H_u \1_{s_j}}^2=\sv_j^2\norm{\1_{s_j}}^2. 
$$
From here we get the inequality
$$
\sum_{j=1}^N \frac{\norm{u_j}^2}{\sv_j^2}=\sum_{j=1}^N\norm{\1_{s_j}}^2\leq 1. 
$$
Now elementary analysis of the equation 
$$
\sum_{j=1}^N \frac{\norm{u_j}^2}{\sv_j^2-x}=1, \quad x\geq0
$$
together with the above inequality 
shows that it has exactly $N$ solutions which are all non-negative and interlace with $\sv_j^2$'s. 
Thus, we obtain that $\wt N=N$ and the interlacing condition \eqref{d1} holds. 
The proof of Theorem~\ref{lma.d1} is complete. 
\qed

\subsection{Proof of Theorem~\ref{inverse}(i)}

1)
By \eqref{n7}, we have 
$$
h_j=s_j^{-1}H_u u_j=s_j^{-1}H_u^2\1_{s_j}=s_j \1_{s_j}.
$$ 
Next, let us write \eqref{a5a} as
$$
H_u(h_j f)=s_j h_j \psi_j\overline{f}, \quad f\in \Ran H_{\psi_j}. 
$$
Applying this to $f=\1$, 
we get
$H_u h_j=s_j h_j\psi_j$, which can be rewritten as
$$
u_j=\psi_j h_j. 
$$
Similarly, applying \eqref{a5b} to the element $\1$, we get
$$
K_u \wt u_k=\svt_k \wt \psi_k \wt u_k, \quad \svt_k>0.
$$ 
We also have $K_u\wt u_N=0$ if $\svt_N=0$.

2) Let us prove \eqref{d7c}. 
Recall the decomposition \eqref{d2a}:
\[
u_j=\| u_j\|^2\sum_{k=1}^N\frac{\wt u_k}{\sv_j^2-\svt_k^2}\ .
\label{e2d}
\]
We would like to apply $H_u$ to both sides of this equation. 
We have, using \eqref{n10}, 
$$
H_u \wt u_k=SS^*H_u\wt u_k+(H_u\wt u_k,\1)\1
=SK_u\wt u_k+(H_u\1,\wt u_k)\1
=\wt s_k S\wt \psi_k\wt u_k+\norm{\wt u_k}^2\1, 
$$
and therefore, using \eqref{d4} at the last step, 
\begin{align}
\sv_jh_j(z)=H_u u_j(z)&=\| u_j\|^2\sum_{k=1}^N\frac{H_u\wt u_k(z)}{\sv_j^2-\svt_k^2}
\notag
\\
&=\| u_j\|^2\left (\sum_{k=1}^N\frac{z\svt_k\wt \psi_k(z)\wt u_k(z)+\norm{\wt u_k}^2}{\sv_j^2-\svt_k^2}\right )\ 
\notag
\\
&=\| u_j\|^2\left (\sum_{k=1}^N\frac{z\svt_k\wt \psi_k(z)\wt u_k(z)}{\sv_j^2-\svt_k^2}+1\right )\ .
\label{e2b}
\end{align}
On the other hand, multiplying \eqref{e2d} by $\sv_j$, we get
\[
\sv_j\psi_j(z)h_j(z)=\sv_j u_j(z)=\| u_j\|^2\sum_{k=1}^N\frac{\sv_j\wt u_k}{\sv_j^2-\svt_k^2}\ .
\label{e2c}
\]
Multiplying \eqref{e2b} by $\psi_j(z)$ and comparing with \eqref{e2c}, we obtain 
$$
\sum_{k=1}^N \frac{\sv_j-z\svt_k \wt \psi_k(z)\psi_j(z)}{\sv_j^2-\svt_k^2}\wt u_k(z)=\psi_j(z);
$$
recalling the vector notation of this theorem, this becomes
$$
\scrC(z)
\wt\bu(z)=\bpsi(z).
$$
This proves \eqref{d7c}.

3)
Let us briefly sketch the proof of \eqref{d7b}, which follows the same logic. 
Applying the operator $SK_u$ to both sides of \eqref{d2b} and observing that
$$
s_j SS^* h_j 
=
s_j h_j -s_j (h_j,\1)\1
=
s_j h_j -(H_u u_j,\1)\1
=
s_j h_j -\norm{u_j}^2\1, 
$$
we obtain (using \eqref{d3} at the last step)
$$
z\wt s_k \wt \psi_k \wt u_k 
=
\norm{\wt u_k}^2 \biggl(\sum_{j=1}^N \frac{s_j h_j}{s_j^2-\wt s_k^2}-1\biggr).
$$
Multiplying \eqref{d2b} by $z\wt s_k\wt \psi_k$ yields
$$
z\wt s_k \wt \psi_k \wt u_k 
=
\norm{\wt u_k}^2 \biggl(\sum_{j=1}^N \frac{z\wt s_k\wt \psi_k\psi_j h_j}{s_j^2-\wt s_k^2}\biggr).
$$
Comparing the last two expressions, we arrive at $\scrC^\top\bh=\bone$, which gives \eqref{d7b}. 

4)
Finally, \eqref{d7} follows from
$$
u=\sum_{k=1}^N \wt u_k\ .
$$
The proof of Theorem~\ref{inverse}(i) is complete.

\section{Surjectivity of the spectral map}\label{sec.C}

\subsection{Preliminaries}

In this section we prove Theorem~\ref{inverse}(ii).
Throughout the rest of the section, we fix  $N\in\bbN$, assume that 
$\{\sv_j\}_{j=1}^N$, $\{\svt_k\}_{k=1}^N$ are real numbers satisfying the interlacing
condition \eqref{d1}, 
and $\{\psi_j\}_{j=1}^N$, $\{\wt\psi_k\}_{k=1}^N$ are arbitrary inner functions.
We let $\scrC(z)$ be the matrix defined by \eqref{d7a} 
and $\bh$, $\wt \bu$, $u$ be defined by \eqref{d7b}, \eqref{d7c}, \eqref{d7}. 
Since $u\in\calH^\infty(\bbT)$, the Hankel operators $H_u$, $K_u$ 
are well defined and bounded. Our aim is to check that $H_u^2$, $K_u^2$ 
have finite spectra and the corresponding spectral data is given by 
$$
\Lambda(u)=\bigl(\{\sv_j\}_{j=1}^N, \{\svt_k\}_{k=1}^N, \{\psi_j\}_{j=1}^N, \{\wt \psi_k\}_{k=1}^N\bigr).
$$

As in Section~\ref{sec.A}, we set 
$$
\calT=\biggl\{\frac{1}{\sv_j^2-\svt_k^2}\biggr\}_{j,k=1}^N. 
$$
We denote by $D(\psi)$ 
the diagonal $N\times N$ matrix $\diag\{\psi_1(z),\dots,\psi_N(z)\}$ and similarly
for $D(\svt)$, $D(\sv)$ etc. 
We also set for brevity
$$
D(\sv\psi):=D(\sv)D(\psi), \quad D(\svt\wt\psi):=D(\svt)D(\wt\psi). 
$$
Note that if $\svt_N=0$, we can choose any inner function for $\wt \psi_N$ in our construction below;  it will always
appear in combination $\svt_N\wt\psi_N$, so this choice is not important. 

With this notation, the matrix $\scrC(z)$ of \eqref{d7a} can be rewritten  (similarly to \eqref{A22}) as
$$
\scrC(z)=D(\sv)\calT-zD(\psi)\calT D(\svt\wt \psi). 
$$
Recall the definitions \eqref{d7b}, \eqref{d7c} of $\bh$, $\wt \bu$: 
these are the unique solutions to the equations
\begin{align}
\scrC(z)^\top \bh&=\bone
\label{e4}
\\
\scrC(z)\wt\bu&=D(\psi)\bone. 
\label{e5}
\end{align}
Uniqueness is guaranteed by Theorem~\ref{thm.z4}. 

We  observe that 
the definition \eqref{d7} of $u(z)$ can be written in 
two alternative ways: 
\begin{align}
u(z)&=\jap{\scrC(z)^{-1}\bpsi(z),\bone}
=\jap{\wt\bu(z),\bone}, 
\label{n3}
\\
u(z)&=\jap{\bpsi(z),\overline{(\scrC(z)^\top)^{-1}\bone}}=\jap{\bpsi(z),\overline{\bh(z)}}=\jap{D(\psi)\bh,\bone}.
\label{n4}
\end{align}

\subsection{The action of $H_u$ and $K_u$}

The heart of the proof is the following algebraic lemma. 
\begin{lemma}\label{lma.e1}
For all $j,k=1,\dots,N$ we have the identities
\begin{align}
H_u(fh_j)=\sv_j\overline{f}\psi_jh_j, \quad \forall f\in \Ran H_{\psi_j}, 
\label{e2}
\\
K_u(g\wt u_k)=\svt_k\overline{g}\wt \psi_k \wt u_k, \quad \forall g\in \Ran H_{\wt \psi_k}.
\label{e3}
\end{align}
If $\svt_N=0$, then \eqref{e3} for $k=N$ should be understood as $K_u\wt u_N=0$. 
\end{lemma}
\begin{proof}
1) 
We recall the crucial identity \eqref{A18}:
$$
D(\sv^2)\calT-\calT D(\svt^2)=\jap{\cdot,\bone}\bone.
$$
Using this identity, as a preparation for the calculations below let us compute
\begin{multline}
\scrC(z)^*D(\sv\psi)+\overline{z}D(\svt\wt\psi)^*\scrC(z)^\top
\\
=
(\calT^\top D(\sv)-\overline{z}D(\svt\wt\psi)^*\calT^\top D(\psi)^*)D(\sv\psi)
+
\overline{z}D(\svt\wt\psi)^*(\calT^\top D(s)-zD(\svt\wt\psi)\calT^\top D(\psi))
\\
=
\calT^\top D(\sv^2)D(\psi)-D(\svt^2)\calT^\top D(\psi)
=(\calT^\top D(\sv^2)-D(\svt^2)\calT^\top)D(\psi)
=\jap{\cdot,D(\psi)^*\bone}\bone.
\label{e11}
\end{multline}

2)
First we aim to prove \eqref{e2} for $f=1$; in our vector notation 
this identity can be written as
\[
\proj(u\overline{\bh})=D(\sv\psi)\bh.
\label{e6}
\]
First, taking the complex conjugate of equation  \eqref{e4} and multiplying by the scalar $u(z)$, we obtain 
$$
\scrC(z)^*(u\overline{\bh})=u\bone. 
$$
Next, using \eqref{e11}, \eqref{e4} and \eqref{n4}, we obtain
\begin{multline*}
\scrC(z)^*D(\sv\psi)\bh
=
\jap{\bh,D(\psi)^*\bone}\bone-\overline{z}D(\svt\wt\psi)^*\scrC(z)^\top\bh
\\
=
\jap{D(\psi)\bh,\bone}\bone-\overline{z}D(\svt\wt\psi)^*\bone
=
u\bone-\overline{z}D(\svt\wt\psi)^*\bone.
\end{multline*}
Putting the last two equations together, we get
$$
\scrC(z)^*(u\overline{\bh}-D(\sv\psi)\bh)=\overline{z}D(\svt\wt\psi)^*\bone.
$$
By Theorem~\ref{thm.z4}, $\scrC(z)$ is invertible and so we can write
$$
u\overline{\bh}-D(\sv\psi)\bh=\overline{z}(\scrC(z)^*)^{-1}D(\svt\wt\psi)^*\bone.
$$
Let us apply the the Szeg\H o projection $\proj$ to both sides of this equation. 
Clearly, the right hand side disappears, and we obtain \eqref{e6}. 

Finally,
from \eqref{e6} we obtain for any
$f\in\Ran H_{\psi_j}$:
\[
H_u(fh_j)
=
\proj(\overline{f}u\overline{h_j})
=
\proj(\overline{f}\proj(u\overline{h_j}))
=
\proj(\overline{f}H_u h_j)=\sv_j\proj(\overline{f}\psi_jh_j)
=\sv_j\overline{f}\psi_jh_j,
\label{e15}
\]
which is exactly \eqref{e2}. 
At the last step, we have used the definition \eqref{n11} of $\Ran H_{\psi_j}$. 
This calculation makes sense for any $f\in\Ran H_{\psi_j}$, since
$h_j\in \calH^\infty(\bbT)$.

3) 
Let us prove \eqref{e3} with $g=1$; this can be written as
\[
K_u\wt\bu=D(\svt\wt \psi)\wt\bu.
\label{e12}
\]
First, taking the complex conjugate of \eqref{e5} and multiplying by $\overline{z}u(z)$, we get
\[
\overline{\scrC(z)}\overline{z}u\overline{\wt\bu}=\overline{z}uD(\psi)^*\bone.
\label{e14}
\]
Next, we transform \eqref{e11} by taking adjoints and multiplying by $\overline{z}$; this gives
$$
\overline{z}D(\sv\psi)^*\scrC(z)+\overline{\scrC(z)}D(\svt\wt\psi)=\overline{z}\jap{\cdot,\bone}D(\psi)^*\bone. 
$$
Using the last equation, \eqref{e5} and \eqref{n3}, we get
\begin{multline*}
\overline{\scrC(z)}D(\svt\wt\psi)\wt\bu
=
-\overline{z}D(\sv\psi)^*\scrC(z)\wt\bu+\overline{z}\jap{\wt\bu,\bone}D(\psi)^*\bone
\\
=
-\overline{z}D(\sv\psi)^*D(\psi)\bone+\overline{z}uD(\psi)^*\bone
=
-\overline{z}D(\sv)\bone+\overline{z}uD(\psi)^*\bone.
\end{multline*}
Subtracting the last equation from \eqref{e14} we obtain
$$
\overline{\scrC(z)}(\overline{z}u\overline{\wt\bu}-D(\svt\wt\psi)\wt\bu)
=\overline{z}D(s)\bone. 
$$
By Theorem~\ref{thm.z4}, $\overline{\scrC(z)}$ is invertible and so we can write
$$
\overline{z}u\overline{\wt\bu}-D(\svt\wt\psi)\wt\bu
=\overline{z}(\overline{\scrC(z)})^{-1}D(s)\bone. 
$$
Let us apply the the Szeg\H o projection $\proj$ to both sides of this equation. 
Clearly, the right hand side disappears, and we obtain \eqref{e12}. 

Finally, 
similarly to \eqref{e15} from here we get for all $g\in \Ran H_{\wt \psi_k}$:
$$
K_u(g\wt u_k)
=
\proj(\overline{g}u\overline{z}\overline{\wt u_k})
=
\proj(\overline{g}\proj(u\overline{z}\overline{\wt u_k}))
=
\proj(\overline{g} K_u \wt u_k)
=
\svt_k \proj(\overline{g}\wt\psi_k\wt u_k)
=
\svt_k \overline{g}\wt \psi_k\wt u_k,
$$
which is the required formula \eqref{e3}.   
\end{proof}

\subsection{Relations between $u_j$ and $\wt u_k$}

\begin{lemma}\label{lma.e2}
Let $u_j=\psi_j h_j$ for all $j$, and let the parameters $\tau_j>0$, $\varkappa_k>0$  be as defined by \eqref{A8}. 
Then for all $j,k$ we have
\[
u_j=\tau_j^2\sum_{k=1}^N\frac{\wt u_k}{\sv_j^2-\svt_k^2},
\quad
\wt u_k=\varkappa_k^2\sum_{j=1}^N\frac{u_j}{\sv_j^2-\svt_k^2}.
\label{e3a}
\]
\end{lemma}
\begin{proof}
First we recall the important identity \eqref{A7}:
\[
\calT^{-1}=D(\varkappa^2)\calT^\top D(\tau^2). 
\label{n1}
\]
From this identity it follows that $\calT^\top D(\tau^2)\calT$ is diagonal, and therefore commutes 
with other diagonal matrices. Using this fact, as a preparation we compute
\begin{multline*}
\calT^\top D(\tau^2)D(\psi)^*\scrC(z)
=
\calT^\top D(\tau^2)D(\psi)^*D(\sv)\calT
-z\calT^\top D(\tau^2)\calT D(\svt\wt\psi)
\\
=
\calT^\top D(\sv)D(\psi)^*D(\tau^2)\calT
-z D(\svt\wt\psi)\calT^\top D(\tau^2)\calT
\\
=
\calT^\top D(\sv)D(\psi)^*D(\tau^2)\calT
-z D(\svt\wt\psi)\calT^\top D(\psi)D(\psi)^* D(\tau^2)\calT
=
\scrC(z)^\top D(\psi)^*D(\tau^2)\calT. 
\end{multline*}
Let us apply the last identity to the element $\wt\bu$:
$$
\calT^\top D(\tau^2)D(\psi)^*\scrC(z)\wt\bu
=
\scrC(z)^\top D(\psi)^*D(\tau^2)\calT\wt\bu.
$$
For the left hand side here, using \eqref{e5} and \eqref{A20a}, we get
$$
\calT^\top D(\tau^2)D(\psi)^*\scrC(z)\wt\bu
=
\calT^\top D(\tau^2)D(\psi)^*D(\psi)\bone
=
\calT^\top D(\tau^2)\bone
=\bone.
$$
Combining the last two equations, we get
$$
\scrC(z)^\top D(\psi)^*D(\tau^2)\calT\wt\bu=\bone.
$$
Now recall that $\bh$ is the unique solution to 
$$
\scrC(z)^\top \bh=\bone.
$$
It follows that 
$$
\bh=D(\psi)^*D(\tau^2)\calT\wt\bu,
$$
which can be rewritten as
\[
D(\psi)\bh=D(\tau^2)\calT\wt\bu.
\label{n2}
\]
Recalling that $u_j=\psi_j h_j$, we see that this is the first one of the  required equations \eqref{e3a}.

Now the second equation in \eqref{e3a} follows by using 
formula \eqref{n1}. 
Indeed, multiplying \eqref{n2} by $D(\varkappa^2)\calT^{\top}$, we obtain 
$$
D(\varkappa^2)\calT^{\top} D(\psi)\bh=\wt\bu,
$$
which is the second equation in \eqref{e3a} written in vector notation. 
\end{proof}

\subsection{Completing the proof}

\begin{lemma}\label{cr.e2}
Let $u$ be given by formula \eqref{d7} of Theorem~\ref{inverse}. 
Then 
$$
\{\sv_j\}_{j=1}^N\subset \Sigma_H(u),\quad 
\{\svt_k\}_{k=1}^N \subset \Sigma_K(u).
$$ 
For all $j$, the function
$u_j\not=0$ is the orthogonal projection of $u$ onto $E_H(\sv_j)$ and  for all $k$, the function
$\wt u_k\not=0$ is the orthogonal projection of $u$ onto $E_K(\svt_k)$.
\end{lemma}
\begin{proof}
1)
Iterating \eqref{e2} with $f=\psi_j$ and \eqref{e3} with $g=1$, we obtain
$$
H_u^2 u_j=\sv_j^2 u_j, 
\quad
K_u^2 \wt u_k=\svt_k^2 \wt u_k;
$$
that is, $u_j\in E_H(\sv_j)$ and $\wt u_k\in E_K(\svt_k)$. 
Since $H_u^2$ and $K_u^2$ are self-adjoint, this means, 
in particular, that all $u_j$ are pairwise orthogonal and similarly all $\wt u_k$ are 
pairwise orthogonal. 

Observe that \eqref{n3}, \eqref{n4} can be written as
$$
u(z)=\sum_{j=1}^N u_j(z)=\sum_{k=1}^N \wt u_k(z). 
$$
It follows that $u_j$ is the orthogonal projection of $u$ onto $E_H(\sv_j)$ and 
$\wt u_k$ is the orthogonal projection of $u$ onto $E_K(\svt_k)$. 

2) 
To complete the proof, it remains to check that $u_j\not=0$ and $\wt u_k\not=0$ for all
$j$, $k$. 
Taking the inner product with $\wt u_k$ in the first equation \eqref{e3a}, we obtain
$$
(\sv_j^2-\svt_k^2)(u_j,\wt u_k)=\tau_j^2\norm{\wt u_k}^2. 
$$
Taking the inner product with $u_j$ in the second equation \eqref{e3a}, we similarly get
$$
(\sv_j^2-\svt_k^2)(u_j,\wt u_k)=\varkappa_k^2\norm{u_j}^2. 
$$
Comparing these two equations yields
$$
\frac{\norm{u_j}^2}{\tau_j^2}
=
\frac{\norm{\wt u_k}^2}{\varkappa_k^2}
$$
for all $j,k$. Since we know that the norms above are non-zero 
at least for some $j,k$, it follows that they are non-zero for all $j,k$. 
(In fact, it is easy to show that $\norm{\wt u_k}^2=\varkappa_k^2$ and 
$\norm{u_j}^2=\tau_j^2$ for all $j,k$, but we don't need this.) 
\end{proof}

We are almost ready to complete the proof of Theorem~\ref{inverse}(ii). 
We need one general lemma:

\begin{lemma}\label{general}
Let $V\subset \calH^2(\bbT)$ be a subspace 
such that $H_u:V\to V$ is onto, $K_u(V)\subset V$, and $V\perp u$. Then $V=\{ 0\}$.
\end{lemma}
\begin{proof}
Given $h\in V$, write $h=H_uh'$ with $h'\in V$. Then $S^*h=K_uh'\in V$. Furthermore, 
$$
0=(h', u)=(h',H_u \1)=
(\1, H_u h')=(\1, h)\ .
$$ 
Hence $S^*(V)\subset V$ and $V\perp \1$. 
Thus we obtain $V=\{0\}$. 
\end{proof}

\begin{proof}[Proof of Theorem~\ref{inverse}(ii)]

1)
First let us consider the subspaces 
\begin{align*}
W_H&=\bigl(\oplus_{j=1}^N E_H(\sv_j)\bigr)\oplus\bigr(\oplus_{k=1}^N E_H(\svt_k)\bigr),
\\
W_K&=\bigl(\oplus_{j=1}^N E_K(\sv_j)\bigr)\oplus\bigr(\oplus_{k=1}^N E_K(\svt_k)\bigr),
\end{align*}
and prove that $W_H=W_K$. 
We know from Lemma~\ref{cr.e2} that 
$u_j$ is the orthogonal projection of $u$ onto $E_H(\sv_j)$ and 
$\wt u_k$ is the orthogonal projection of $u$ onto $E_K(\svt_k)$. 
By Proposition~\ref{lma.z1}, we have
\begin{align*}
E_H(\sv_j)&=\Span\{u_j\}\oplus E_K(\sv_j),
\\
E_K(\svt_k)&=\Span\{\wt u_k\}\oplus E_H(\svt_k)
\end{align*}
for each $j$ and $k$. It follows that 
\begin{align*}
W_H
&=
\bigl(\oplus_{j=1}^N \Span\{u_j\}\bigr)
\oplus
\bigl(\oplus_{j=1}^N E_K(\sv_j)\bigr)
\oplus
\bigl(\oplus_{k=1}^N E_H(\svt_k)\bigr),
\\
W_K
&=
\bigl(\oplus_{k=1}^N \Span\{\wt u_k\}\bigr)
\oplus
\bigl(\oplus_{j=1}^N E_K(\sv_j)\bigr)
\oplus
\bigl(\oplus_{k=1}^N E_H(\svt_k)\bigr).
\end{align*}
Relations \eqref{e3a} mean that 
$$
\oplus_{j=1}^N \Span\{u_j\}=\oplus_{k=1}^N \Span\{\wt u_k\}.
$$
It follows that $W_H=W_K$. 

2) 
Let us check that 
$$
\Ran H_u=W_H,
$$
where $W_H$ is the subspace defined on the previous step. 
For $\eps>0$, let 
$$
W_\eps=\Ran \chi_{(\eps,\infty)}(\abs{H_u}),
$$
where $\chi_{(\eps,\infty)}$ is the characteristic function 
of the interval $(\eps,\infty)$. 
It is clear that $W_\eps\subset\Ran H_u$ and that 
$$
\overline{\Ran H_u}=\cup_{\eps>0} W_\eps.
$$
Further, for a sufficiently small $\eps>0$
(such that  $\eps<\svt_N$ if $\svt_N>0$ and 
$\eps<\sv_{N}$ if $\svt_N=0$)
 we have $W_H\subset W_\eps$. 
Let us prove that for such $\eps$,  the subspace
$$
V=W_\eps\cap W_H^\perp
$$
is trivial. 
We aim to use Lemma~\ref{general}. 

Since $H_u$ commutes with $\chi_{(\eps,\infty)}(\abs{H_u})$, 
it is clear that $H_u(W_\eps)=W_\eps$. 
Also, by the definition of $W_H$, we have $H_u(W_H)=W_H$.  
Thus we obtain  $H_u(V)=V$. 

Let us check that $K_u(V)\subset V$. 
First note that $u\in W_H$ and so $V\perp u$. 
By \eqref{a5}, 
it follows that for $f\in V$, we have 
$$
\varphi(H_u^2) f=\varphi(K_u^2) f
$$ 
for all bounded functions $\varphi$. Next, $K_u(W_K)\subset W_K$ and so 
if $f\in V$, then $K_u f\perp W_K$ and therefore $K_u f\perp u$. 
Thus, by the same logic  we obtain 
$$
\varphi(H_u^2) K_u f=\varphi(K_u^2) K_uf.
$$ 
Thus, for $f\in V$ we obtain
\begin{multline*}
\chi_{(\eps,\infty)}(\abs{H_u})K_u f
=
\chi_{(\eps,\infty)}(\abs{K_u})K_uf
\\
=
K_u\chi_{(\eps,\infty)}(\abs{K_u})f
=
K_u\chi_{(\eps,\infty)}(\abs{H_u})f
=
K_uf,
\end{multline*}
and so $K_uf\subset W_\eps$. It follows that $K_u(V)\subset V$, as claimed. 

Now applying Lemma~\ref{general}, we obtain $V=\{0\}$, and so 
$$
W_\eps=W_H
$$
for all sufficiently small $\eps>0$. It follows that $\Ran H_u$ is closed and coincides with $W_H$. 

3)
From the previous step and from Lemma~\ref{cr.e2} it follows that 
the spectrum of $H_u^2$ is finite and 
$$
\Sigma_H(u)=\{\sv_j\}_{j=1}^N.
$$
By Theorem~\ref{lma.d1}, it follows that $\Sigma_K(u)$ consists of $N$ elements. 
On the other hand, again by Lemma~\ref{cr.e2}, we know that 
$$
\{\svt_k\}_{k=1}^N\subset\Sigma_K(u); 
$$
thus, in fact we have $\{\svt_k\}_{k=1}^N=\Sigma_K(u)$. 
By Lemma~\ref{lma.e1}, the inner functions entering the spectral data for $u$ are
precisely $\{\psi_j\}$, $\{\wt \psi_k\}$. 
The proof of Theorem~\ref{inverse}(ii) is complete. 
\end{proof}


\end{document}